\documentclass[11pt]{amsart}
\usepackage{amssymb,latexsym, amsmath, amsxtra}
\usepackage[dvips]{graphics}
\def\sumstar{\operatornamewithlimits{\sum\nolimits^*}}

\theoremstyle{plain}
\newtheorem{theorem}{Theorem}
\newtheorem{corollary}[theorem]{Corollary}

\newtheorem{remark}[theorem]{Remark}
\newtheorem{lemma}[theorem]{Lemma}

\theoremstyle{definition}

\theoremstyle{remark}

\numberwithin{equation}{section}

\begin{document}
\title{The mean-square of Dirichlet L-functions}
\newcommand{\sumprime}{\operatornamewithlimits{\sum\nolimits^\prime}}
\newcommand{\Res}{\operatornamewithlimits{\text{Res}}}
\newcommand{\sumeven}{\operatornamewithlimits{\sum\nolimits^e}}
\newcommand{\sumodd}{\operatornamewithlimits{\sum\nolimits^o}}

\author {J.B. Conrey}
\address{American Institute of Mathematics,
360 Portage Ave, Palo Alto, CA 94306} \address{School of
Mathematics, University of Bristol, Bristol, BS8 1TW, United
Kingdom} \email{conrey@aimath.org}

\thanks{
Research  supported by the American Institute
of Mathematics and a Focused Research Group grant from the
National Science Foundation.    This paper was started when the author was
visiting the Isaac Newton Institute for Mathematical Sciences for
the programme ``Random matrix approaches  in number theory".
 }
\maketitle

\section{Introduction}

In this note we find an asymptotic formula for the mean-square of
primitive Dirichlet L-functions near 1/2:
$$S_q(\alpha,\beta):=\sumstar_{\chi\bmod q} L(1/2+\alpha,\chi)L(1/2+\beta,\overline{\chi})$$
where $\alpha $ and $\beta$ are small complex numbers satisfying
$\alpha,\beta \ll 1/\log q$.

Ingham [Ing] considered an analogous moment for the Riemann zeta-function on the critical line
with small shifts. Paley [Pal]  considered 
the moment above for Dirichlet $L$-functions.  
Heath-Brown [HB] has computed a similar moment, but for all characters modulo $q$, in the case that
$\alpha=\beta=0$. His result is
\begin{theorem}[HB]There are constants $c_\ell$ such that
\begin{eqnarray*}
\sum_{\chi \bmod q} |L(1/2,\chi)|^2=\frac{\phi(q)}{q}\sum_{k\mid q} \mu(q/k)T(k),
\end{eqnarray*}
with
\begin{eqnarray*}
T(k)=k\log k +A k +B\sqrt{k}+
\sum_{\ell=0}^{2L-1}c_\ell k^{-\ell/2}
+O(k^{-L})
\end{eqnarray*}
where
\begin{eqnarray*}
A=\gamma -\log 8\pi = 2\gamma-\log  2\pi+\frac{\Gamma'}{\Gamma}(1/2) \qquad \mbox{ and } \qquad B= 2\zeta(1/2)^2;
\end{eqnarray*}
 $L$ is any positive integer.
\end{theorem}

Note that there are main terms of size $k\log k$ for all divisors $k$ of $q$. By contrast, the conjectures of [CFKRS]
predict a very specific main term, with a square root size error-term, for any family of $L$-functions.
The recipe given in that paper leads to the prediction, for example, that
\begin{eqnarray*}
&&\sumstar_{\chi\bmod q\atop \chi(-1)=1}
L(1/2+\alpha,\chi)L(1/2+\beta,\overline{\chi})\\
&& \qquad =\frac{\phi^*(q)}{2}(\zeta_q(1+\alpha+\beta)   +
X^+(q,\alpha,\beta)\zeta_q(1-\alpha-\beta))+O(q^{1/2+\epsilon}).
  \end{eqnarray*}
Here $X^+$ is related to the factor from the functional equation given below in (\ref{eq:X}), $\phi^*(q)/2+O(1)$ is the number
of primitive even characters modulo $q$ and
\begin{equation*}
\zeta_q(s)=\sum_{n=1\atop (n,q)=1}^\infty \frac{1}{n^s}=\prod_{p\mid q}\left(1-\frac{1}{p^s}\right)\zeta(s).
\end{equation*}
 Theorem 5 proves that this prediction is correct.
In particular, there are no terms between those of size $\phi^*(q)$ and the error-term
of size $q^{1/2+\epsilon}$ in contrast to Heath-Brown's result. The point is that it is the primitive characters
which form a family in the sense defined in [CFKRS], and the set of all characters does not.

In the case that $q$ is prime, the set of all characters consists of the primitive characters with just one
additional non-primitive character, namely the trivial character.
In this case, Heath-Brown's result gives
an asymptotic series in powers of $q^{-1/2}$ for the above sum.
An interesting feature is that there is a main-term of size $q^{1/2+\epsilon}$. Such a main-term could
never be obtained from
the random matrix considerations in [CFKRS].
We find that this term of size $q^{1/2}$ is present for composite $q$ as well; see the statement of Theorem 5.

Heath-Brown compares the behavior of the mean-square over all $L$-functions with that of the
Riemann zeta function. He remarks that the above sum is more like
the smoothly weighted
$$\int _0^\infty |\zeta(1/2+it)|^2 e^{-\delta t}~dt,$$
for which one obtains an asymptotic series in powers of $\delta$,
than it is like
$$\int_0^T |\zeta(1/2+it)|^2~dt$$
for which one expects an error term of size $T^{1/4+\epsilon}$.
In this note we find that upon considering all $q$ that the behavior does in fact have elements
in common with the
second mean-value. In particular, for certain $q$, namely $q=p_1 p_2$ with $p_1$ and $p_2$ primes near
the square-root of $q$, we obtain a formula with an error term
of size $q^{1/4}$. For other $q$, the error term can be slightly
larger; see Corollary 6.

Finally, we develop an asymptotic formula for
$$S(p,h) := \sumstar_{\chi \bmod p}|L(1/2,\chi)|^2\chi(h)$$
in the case that both $h$ and $p$ are primes with $0<h<p$; see Theorem 10.
This formula is more accurate than in any previous work; also, it reveals a
connection with $S(h,-p)$.

The author would like to thank Matthew Young for helpful conversations about this work.
\section{Initial considerations}

Our treatment at the start closely follows the analysis of [HB].
It is convenient to split our sum
into a sum over even characters and a sum over odd characters. To
this end let
$$S_q^+(s,\alpha,\beta):=  \sumstar_{\chi\bmod q\atop \chi(-1)=1}
L(s+\alpha,\chi)L(1-s+\beta,\overline{\chi})$$
where the sum is over primitive characters modulo $q$. Then we wish to evaluate
$S_q^+(1/2,\alpha, \beta)$. We assume that
$q>1$ so that $S_q^+(s,\alpha, \beta)$ is an entire function of
$s$. Initially, we regard $s$ as a complex number with $\Re s>2$;
the formula we eventually develop will hold for arbitrary $s$
whereupon we will we take $s=1/2$. Note, for future reference,
that
$$S_q^+(1-s,\alpha,\beta)=S_q^+(s,\beta,\alpha).$$

\medskip

\begin{lemma}[Functional equation] If $\chi$ is a primitive character
modulo $q$, then
$$L(1-s,\overline{\chi})=X(1-s,\overline{\chi})L(s,\chi)$$
where
\begin{equation} \label{eq:X} X(1-s,\overline{\chi})=\tau(\overline{\chi}) q^{s-1}(2\pi)^{-s}\Gamma(s)
\big( e^{-\frac{\pi i s}{2}}+\chi(-1)e^{\frac{\pi i s}{2}}\big).
\end{equation}
\end{lemma}
As usual,   $\tau$   denotes the Gauss sum
$$\tau(\chi)=\sum_{a=1}^q \chi(a) e(a/q)$$
where $e(x)=e^{2\pi i x}$.  Note that for even characters, the
last factor is $2\cos \tfrac{\pi}{2} s$ and for odd characters it
is $-2i\sin \tfrac{\pi}{2} s$.
Davenport [D] is a good reference for these facts.

Applying this lemma, we find that $S_q^+(s,\alpha, \beta)=$
$$\frac{2\Gamma(s-\beta)\cos \tfrac{\pi}{2} (s-\beta)}{q}  \sumstar_{\chi
\bmod q\atop \chi(-1)=1} \sum_{a=1}^q
\overline{\chi}(a) e\big(\frac aq\big ) \sum_{m,n=1}^\infty
m^{-\alpha-\beta} \chi(mn)   \big(\frac{2\pi mn}
{q}\big)^{-s+\beta}   .$$

\begin{lemma} [Orthogonality relation for primitive characters] If
$(mn,q)=1$, then
$$\sumstar_{\chi \bmod q} \chi(m)\overline{\chi}(n)=\sum_{ d\mid q\atop
d\mid m-n} \phi(d)\mu\big(\frac q d \big).$$
\end{lemma}

As a corollary, if $(mn,q)=1$, then
$$\sumstar_{\chi \bmod q\atop \chi(-1)=1} \chi(m)\overline{\chi}(n)=
\tfrac 12 \sum_{ d\mid q\atop d\mid m \pm n} \phi(d)\mu\big(\frac
q d \big).$$ Here $d\mid m\pm n$ means that we include terms for
which either $d\mid (m-n)$ or $d\mid (m+n)$.

 Applying Lemma 2 leads to
\begin{eqnarray*}S_q^+(s,\alpha,\beta)&=&  \frac{\Gamma(s-\beta)\cos
\tfrac{\pi}{2} (s-\beta)}{q}
\sum_{a=1\atop (a,q)=1}^q e\big(\frac a{q}\big ) \sum_{cd=q} \mu(c)\phi(d)
 \sum_{mn\equiv \pm a\bmod d  \atop (mn,q)=1 }
m^{-\alpha-\beta} \big(\frac{2\pi mn}{q}\big)^{-s+\beta}
 .\end{eqnarray*}
Now we obtain an analytic continuation for $S_q(s,\alpha,\beta)$.
Let
$$F(s,\alpha,\beta)=\frac{e^{s^2}}{s}\frac{\cos \pi (s+\alpha)}{\cos \pi
\alpha}\frac{\cos \pi
(s-\beta)}{\cos \pi \beta}.$$ Clearly, $F$ is analytic everywhere
apart from a simple pole at $s=0$ with residue 1. Moreover,
$$F(s,\alpha, \beta)= -F(-s,\beta,\alpha).$$ Using this equation
together with the functional equation for $S$ and Cauchy's
theorem, we obtain
\begin{eqnarray*} \frac{1}{2\pi i} \int_{(2)} F(s,\alpha,\beta)
S_q^+(s+1/2,\alpha,\beta) ~ds&=& S_q^+(1/2,\alpha,\beta)\\
\qquad && - \frac{1}{2\pi i} \int_{(2)} F(s,\beta, \alpha)
S_q^+(s+1/2,\beta, \alpha)~ds
\end{eqnarray*}
In other words, $$S_q^+(1/2,\alpha,\beta)
=M(\alpha,\beta)+M(\beta,\alpha)$$ where
$$M(\alpha,\beta)=  \frac{1}{2\pi i} \int_{(2)} F(s,\alpha,\beta)
S_q^+(s+1/2,\alpha,\beta) ~ds. $$

To simplify the notation here let
\begin{eqnarray}\label{eqn:Khat}  K_{\alpha,\beta}(x)&=&\frac{1}{2\pi i} \int_{(2)} F(s,\alpha,
\beta)\Gamma(s+\frac
12-\beta)\cos \tfrac{\pi}{2}(s+\tfrac 12-\beta)
x^{-s-\tfrac{1}{2}+\beta} ~ds\\
&=&\nonumber \frac{x^{\beta-1/2}}{2\pi i}
\int_{(2)} \hat{K}_{\alpha,\beta}(s)x^{-s}~ds
\end{eqnarray}
where
\begin{eqnarray*}
\hat{K}_{\alpha,\beta}(s)=\frac{e^{s^2}}{s}\frac{\cos \pi (s+\alpha)}{\cos \pi
\alpha}\frac{\cos \pi
(s-\beta)}{\cos \pi \beta}\Gamma(s+\frac
12-\beta)\cos \tfrac{\pi}{2}(s+\tfrac 12-\beta).
\end{eqnarray*}
Note that $\hat{K}_{\alpha,\beta}(s)$ has a simple pole at $s=0$ and no other poles.
Hence,  if $x>1$, then by
moving the path of integration to the right we find that
$K_{\alpha, \beta}(x)\ll x^{-N}$ for any $N$. If $x<1,$ then by
moving the path to the left we see that $K_{\alpha,\beta}(x)=
 \Gamma(1/2-\beta)\cos\frac{\pi}{2}(\tfrac
12-\beta) x^{-\tfrac 12 +\beta} +O(x^N)$ for any $N$:
\begin{eqnarray}  \label{eq:K}
K_{\alpha,\beta}(x)=\left\{\begin{array}{ll}O(x^{-N}) &\mbox{ if $x>1$}\\
\Gamma(1/2-\beta)\cos\frac{\pi}{2}(\tfrac
12-\beta) x^{-\tfrac 12 +\beta} +O(x^N) &\mbox{ if $x<1$}
\end{array}\right.
\end{eqnarray}

Integrating term-by-term
\begin{eqnarray*}
M(\alpha,\beta)&=& \frac{1}{q} \sum_{a=1\atop (a,q)=1}^q e\big(\frac
a{q}\big ) \sum_{cd=q} \mu(c)\phi(d)
 \sum_{mn\equiv \pm a\bmod d \atop (mn,q)=1  } m^{-\alpha-\beta}
K_{\alpha,\beta}\big(\frac{2\pi  m n}{q}\big)  \end{eqnarray*}

Bringing the sum over $a$ to the inside gives:

\begin{eqnarray*}
M(\alpha,\beta)&=& \frac{1}{q}   \sum_{cd=q} \mu(c)\phi(d)
 \sum_{(mn,cd)=1}m^{-\alpha-\beta}
K_{\alpha,\beta}\big(\frac{2\pi  m n}{q}\big)
 \sum_{1\le a \le q, (a,q)=1 \atop a\equiv \pm mn \bmod d   } e\big(\frac
a{q}\big )
\end{eqnarray*}

The sum over $a$ may be evaluated by the following lemma.

\begin{lemma}[An exponential sum] Suppose that $(r,d)=1$. If $(c,d)=1$, then
$$ \sumprime_{a=1\atop a\equiv r \bmod d}^{cd}  e\big(\frac a{cd}\big )
=\mu(c)  e\big(\frac {-r\overline{d}}{c}\big )
 e\big(\frac r{cd}\big )= \mu(c)e\big(\frac {r\overline{c}}{d}\big )
$$
where the $\sumprime$ indicates that the sum is for $(a,cd)=1$ and $\overline{d}$ is the inverse of $d$ modulo $c$. If
$(c,d)>1$ then the above sum is 0.
\end{lemma}

\begin{proof} Write $a=r+dx$ where $x$ ranges from 1 to $c$. The
sum in question is
$$ e\big(\frac{r}{cd}\big) \sum_{x=1 \atop (r+dx,cd)=1}^c
e\big(\frac{x}{c}\big) .$$ Since $(r,d)=1$ the condition $(r+dx
,d)=1$ is automatically satisfied so that the condition on $x$
simplifies to $(r+dx,c)=1$. Suppose that $(c,d)=1$. Then we can
make the change of variable $y=r+dx$; this means that
$x=\overline{d}(y-r)$. Then the sum in question can be written as
$$ e\big(\frac{r}{cd}\big) e\big(\frac{-r\overline{d}}{c}\big)\sum_{y=1
\atop (y,c)=1}^c
e\big(\frac{\overline{d}y}{c}\big) .$$ The sum over $y$ is a
Ramanujan sum which evaluates to $\mu(c)$ as desired. In the case
that $(c,d)=g>1$ we write $c=gc'$ and $d=gd'$. The coprimality
condition on $x$ then becomes $(r+gd'x,gc')=1$. If this condition
is satisfied for $x=x_0$ then it is also satisfied for $x=x_0+\ell
c'$ for $1\le \ell \le g$. But this subsum over $x\equiv x_0 \bmod
c'$ is a geometric progression which sums to 0. The total sum is
made up of a collection of such subsums and so it is equal to 0 as
well.
\end{proof}

Applying this lemma to $M$ we find that

\begin{eqnarray*}
M(\alpha,\beta)&=& \frac{1}{q}   \sum_{cd=q\atop (c,d)=1} \mu^2(c)\phi(d)
 \sum_{(mn,q)=1}m^{-\alpha-\beta}
K_{\alpha,\beta}\big(\frac{2\pi  m n}{q}\big)
\left(e\big(\frac
{mn\overline{c}}{d}\big )+e\big(\frac
{-mn\overline{c}}{d}\big )\right)
\end{eqnarray*}

Now we introduce a parameter $D$ and split $M$ into two different sums according to whether
$d\le D$ or $d>D$. Let $M_L(\alpha,\beta)$ denote the terms of
$M(\alpha,\beta)$ for which $d\le D$ and let $M_U(\alpha, \beta)$
denote those terms with $d>D$.

To estimate $M_L$ we replace $K$ by its Mellin transform and have
\begin{eqnarray}\label{eq:MLanalysis}
M_L(\alpha,\beta)&=& \frac{2}{q}   \sum_{cd=q, d\le D \atop (c,d)=1} \mu^2(c)\phi(d)
 \frac{1}{2\pi i}
\int_{(2)} \hat{K}_{\alpha,\beta}(s)\left(\frac{q}{2\pi}\right)^{s+1/2-\beta}
D_q(s+1/2;\alpha,-\beta;\frac{\overline{c}}{d})~ds
\end{eqnarray}
where
\begin{eqnarray} \label{eq:D}
D_q(s;\alpha,\beta;r/d):=\sum_{(mn,q)=1}m^{-s-\alpha}n^{-s-\beta}
\cos \tfrac{2\pi mnr}{d}.
\end{eqnarray}

Now we  can write
\begin{eqnarray}\label{eq:D1}
D_q(s,\alpha,\beta;r/d)=\frac{1}{\phi(d)}\sum_{\psi \bmod d}\tau(\overline{\psi})\tfrac 12(\psi(r)+\psi(-r))
L_q(s+\alpha,\psi)L_q(s+\beta,\psi).
\end{eqnarray}
The only terms with poles arise from $\psi=\psi_0$ the principal character,
for which we have $L_q(s,\psi_0)=\zeta_q(s)$ and $\tau(\psi_0)=\mu(d)$.

We move the path of integration in the $s$-integral to $\Re s=-1/2$. In doing so, we cross poles at
$s=1/2-\alpha$, at $s=1/2+\beta$ and at $s=0$.
Now, it can be shown that
\begin{eqnarray} \label{eq:f}
\sum_{\psi \bmod c} \int_{-\infty}^\infty |f(it)||L(it,\psi)|^2~dt
\ll \phi(c)c
\end{eqnarray}
for any smooth rapidly decaying function $f$.
 Since   $\tau(\psi)\ll d^{\tfrac
12}$, the integral on the new path is $\ll d^{3/2}$, and the error-term is
$$\ll \frac{1}{q}\sum_{cd=q, (c,d)=1 \atop
d\le D} \phi(d)d^{3/2}
.$$

In this way, we find that
\begin{eqnarray}  \label{eq:ML}
&&M_L(\alpha,\beta)=
\frac{2}{q}   \sum_{cd=q, d\le D \atop (c,d)=1} \mu^2(c)\phi(d)\bigg(
\frac{\mu(d)}{\phi(d)}\frac{\phi(q)}{q}\bigg(\hat{K}_{\alpha,\beta}(1/2-\alpha)\left(\frac{q}{2\pi}\right)^{1-\alpha-\beta}
\zeta_q(1-\alpha-\beta)
\\
&&\qquad \qquad \qquad \qquad \qquad \nonumber +\hat{K}_{\alpha,\beta}(1/2-\beta)\left(\frac{q}{2\pi}\right)
\zeta_q(1+\alpha+\beta)\bigg)\\
&&\qquad \qquad \qquad \nonumber  +2\Gamma(1/2-\beta)\cos\tfrac \pi 2 (1/2-\beta)
\left(\frac{q}{2\pi}\right)^{1/2-\beta}
D_q(1/2,\alpha,-\beta,\overline{c}/d)\bigg)
\\
&&\quad \qquad \qquad \qquad \qquad \qquad \nonumber +O\bigg(\frac{1}{q}\sum_{cd=q, (c,d)=1 \atop
d\le D} \phi(d)d^{3/2} \bigg)
\end{eqnarray}

The first two terms of the above can be estimated by $\tau(q)\ll q^{\epsilon}$. Also,
we see from (\ref{eq:D1}) that
\begin{eqnarray*}
D_q(1/2,\alpha,-\beta,\overline{c}/d)
=\frac{1}{2\phi(d)}\sumeven_{\psi \bmod d}\tau(\overline{\psi})\overline{\psi}(c)
L_q(1/2+\alpha,\psi)L_q(1/2-\beta,\psi).
\end{eqnarray*}

Now we turn to $M_U$. We have
  $M_U(\alpha,\beta)=$
\begin{eqnarray*}
 &&\frac{1}{q} \sum_{cd=q\atop {(c,d)=1  \atop
d>D}} \mu^2(c)\phi(d)
 \sum_{m,n \atop (mn,q)=1} m^{-\alpha-\beta}
K_{\alpha,\beta}\big(\frac{2\pi  m n}{q}\big) \bigg(e\big(\frac
{-mn\overline{d}}{c}\big )
 e\big(\frac {mn}{q}\big ) +e\big(\frac
{mn\overline{d}}{c}\big )
 e\big(\frac {-mn}{q}\big )
\bigg)\end{eqnarray*}

Now, since $(mnd,c)=1$, we can write $e(-mn\overline{d}/c)$ in
terms of characters modulo $c$:
$$ e\big(\frac {-mn\overline{d}}{c}\big )=\frac{1}{\phi(c)}\sum_{\psi
\bmod c}
 \tau(\overline{\psi})\psi(-mn\overline{d}).$$
Thus, we obtain
\begin{eqnarray*}
M_U(\alpha,\beta)&=& \frac{1}{q} \sum_{cd=q\atop {(c,d)=1 \atop
d>D}} \frac{\mu^2(c)\phi(d)}{\phi(c)} \sum_{\psi \bmod c}
\tau(\overline{\psi})\psi(- \overline{d})\\
&&\qquad \times
  \sum_{m,n \atop (mn,d)=1}\psi(mn)  m^{-\alpha-\beta}
K_{\alpha,\beta}\big(\frac{2\pi  m n}{q}\big)
 \bigg( e\big(\frac {mn}{q}\big ) +\psi(-1)
 e\big(\frac {-mn}{q}\big )
\bigg)\end{eqnarray*}

Now write
$$e^{\pm ix} K_{\alpha,\beta}(x)=\frac{1}{2 \pi i} \int_{(2)}
\mathcal{K}^{\pm}_{\alpha,\beta}(s)x^{-s}~ds$$ where
$$\mathcal{K}^{\pm}_{\alpha,\beta}(s)=\int_0^\infty e^{\pm ix}
K_{\alpha,\beta}(x)x^s
\frac{dx}{x}.$$
Repeated integration by parts shows that the function $\int_0^1 e^{ix} x^s \tfrac{dx}x$ has poles
at $s=0,-1,-2,\dots.$ Thus, by use of (\ref{eq:K}), we see that
$\mathcal{K}^{\pm}_{\alpha,\beta}(s)$ is analytic apart from simple poles
at $s=1/2-\beta, -1/2-\beta, -3/2-\beta,\dots.$
Let $r^{\pm}_{\alpha,\beta}(\ell)$ denote the residue at $s=1/2-\beta-\ell$, for $\ell=0,1,2,\dots$.
 Then,
 \begin{equation} r^{+}_{\alpha,\beta}(0)=r^{-}_{\alpha,\beta}(0)=:
r_{\alpha, \beta}=\Gamma(1/2-\beta)\cos\tfrac{\pi}{2}(\tfrac12-\beta).
\end{equation}
 It is easy to see that
$\mathcal{K}^{\pm}_{\alpha, \beta}(s)$ decays rapidly in vertical
strips which avoid the poles.
(In general, if $K(x)$ is infinitely differentiable and satisfies, for any $n\ge0$ and any fixed $N>0$,
\begin{eqnarray}  \label{eq:Kcond}
K^{(n)}(x)=\left\{ \begin{array}{ll}\delta_n(x)+O(x^N) &\mbox{ if $0<x<1$ }\\
O(x^{-N}) & \mbox{ if $x>1$ }
\end{array}\right.
\end{eqnarray}
where $\delta_n(x)=0$ if $n>0$ and $\delta_0(x)$ is a nice smooth function, then
\begin{eqnarray*} \hat{K}(s)=\int_0^\infty K(x) x^s\frac{dx}{x}
\end{eqnarray*}
satisfies
\begin{eqnarray} \label{eq:Kestimate}
\hat{K}(s)\ll \max_{x}|K^{(N)}(x)||s|^{-N}
\end{eqnarray}
as $|\Im s| \to \infty$ for any fixed $N>0$. The proof is as follows. If $\Re s>0$ then repeated integration by parts
gives
\begin{eqnarray*}
\hat{K}(s)=(-1)^{n+1}\int_0^\infty K^{(n+1)}(x)\frac{x^{s+n}}{s(s+1)\dots (s+n)}~dx,
\end{eqnarray*}
the integrated terms vanishing at 0 because of the $x^{s+n}$ factor; then (\ref{eq:Kestimate}) follows easily in this
situation.
Another formula for $\hat{K}(s)$ is obtained by replacing $x$ by $1/x$:
\begin{eqnarray*}
\hat{K}(s)=\int_0^\infty K(1/x) x^{-s}\frac{dx}{x};
\end{eqnarray*}
we are still assuming that $\Re s>0$. If we integrate by parts, we have
\begin{eqnarray*}
\hat{K}(s)&=&K(1/x)\frac{x^{-s}}{-s}\bigg|_{0}^\infty-\frac{1}{s}\int_0^\infty K'(1/x) x^{-2-s}~dx\\
&=& -\frac{1}{s}\int_0^\infty K'(1/x) x^{-2-s}~dx;
\end{eqnarray*}
this formula is now valid for $\Re s>-1$. Repeatedly integrating by parts will lead us to (\ref{eq:Kestimate}).)

We have
\begin{eqnarray} \label{eq:MUint}
M_U(\alpha,\beta) &=& \frac{1}{q}\sum_{cd=q\atop {(c,d)=1\atop
d>D} } \frac{\mu^2(c)\phi(d)}{\phi(c)}\sum_{\psi \bmod
c}\tau(\overline{\psi})\overline{\psi}(-d)\\
&&\times  \frac{1}{2\pi
i}\int_{(2)}\big(\mathcal{K}^+_{\alpha,\beta}(s)+\psi(-1)\mathcal{K}^-_{\alpha,\beta}(s)\big)
\frac{q^s}{(2\pi)^s} \sum_{m,n \atop (mn,d)=1}
\frac{\psi(mn)m^{-\alpha-\beta}}{m^sn^s} ~ds \nonumber
\end{eqnarray}

The sum over $m$ and $n$  is
$$\sum_{m,n \atop (mn,d)=1}
\frac{\psi(mn)m^{-\alpha-\beta}}{m^sn^s}=L_d(s,\psi)L_d(s+\alpha+\beta,\psi)
 .
$$

When $\psi=\psi_0$ is the principal character modulo $c$ this sum
has a pole at $s=1$ and a pole at $s=1-\alpha-\beta$ and no other
singularities.  If $\psi$ is not a principal character then the
sum is entire.

Let $\mathcal{K}_{\alpha,\beta}(s)=\mathcal{K}^+_{\alpha,\beta}(s)+\mathcal{K}^-_{\alpha,\beta}(s)$.
We move the path of integration to the vertical line through 0; in
doing so we pass the poles at $s=1$ and at $s=1-\alpha -\beta$
from the principal characters and also the pole at $s=1/2-\beta$
from $\mathcal{K}_{\alpha, \beta}(s)$.

Now, invoking (\ref{eq:f}), and since
  $\tau(\psi)\ll c^{\tfrac
12}$, the integral on the new path is
$$\ll \frac{1}{q}\sum_{cd=q \atop
d>D} \phi(d)c^{3/2}= \sqrt{q}\sum_{d\mid q \atop
d>D} \frac{\phi(d)}{d^{3/2}}
.$$

Thus,
\begin{eqnarray*} M_U(\alpha, \beta)&=&\frac{1}{q}\sum_{cd=q\atop
{(c,d)=1\atop d>D} }
\frac{\mu(c)\phi(d)}{\phi(c)}\bigg(\mathcal{K}_{\alpha,\beta}(1)\frac{q}{2\pi}\zeta(1+\alpha+\beta)
\prod_{p\mid q} \big( 1-\frac{1}{p}\big)\big(
1-\frac{1}{p^{1+\alpha+\beta}}\big) \\
&&+ \mathcal{K}_{\alpha,\beta}(1-\alpha-\beta)
\big(\frac{q}{2\pi}\big)^{1-\alpha-\beta}\zeta(1-\alpha-\beta)
\prod_{p\mid q} \big( 1-\frac{1}{p}\big)\big(
1-\frac{1}{p^{1-\alpha-\beta}}\big)  \\
&&\quad   \qquad+  2 r_{\alpha, \beta}\mu(c)\sumeven_{\psi \bmod
c}\tau(\overline{\psi})\overline{\psi}(-d)
\big(\frac{q}{2\pi}\big)^{\tfrac 12 -\beta} L_q(\tfrac
12-\beta,\psi)L_q(\tfrac 12 +\alpha,\psi) \\
&& \quad \qquad \qquad
 +O\bigg( \sqrt{q}\sum_{d\mid q \atop
d>D} \frac{\phi(d)}{d^{3/2}} \bigg)\\
\end{eqnarray*}

We can simplify this expression a little by noting that
$$\frac{1}{q}\sum_{cd=q\atop {(c,d)=1\atop d>D} }
\frac{\mu(c)\phi(d)}{\phi(c)}=\frac{\phi(q)}{q}\sum_{cd=q\atop
{(c,d)=1\atop c\le \frac{q}{D}} } \frac{\mu(c) }{\phi(c)^2} =
\frac{\phi(q)}{q}\sum_{cd=q\atop (c,d)=1  } \frac{\mu(c)
}{\phi(c)^2}+O(q^{\epsilon-2}D^2). $$ Upon using $\prod_{p\mid
q}(1-1/p)=\phi(q)/q$ we see that the first main term of $M_U$ is
\begin{eqnarray*}
&& \frac{\phi(q)^2}{q^2}\sum_{cd=q\atop (c,d)=1  } \frac{\mu(c)
}{\phi(c)^2}\bigg(\mathcal{K}_{\alpha,\beta}(1)\frac{q}{2\pi}\zeta(1+\alpha+\beta)
\prod_{p\mid q}  \big(
1-\frac{1}{p^{1+\alpha+\beta}}\big)\\
&&\qquad + \mathcal{K}_{\alpha,\beta}(1-\alpha-\beta)
\big(\frac{q}{2\pi}\big)^{1-\alpha-\beta}\zeta(1-\alpha-\beta)
\prod_{p\mid q} \big( 1-\frac{1}{p^{1-\alpha-\beta}}\big)\bigg)
+O(q^{\epsilon-1}D^2).
\end{eqnarray*}

Altogether now we have
\begin{eqnarray*}
M(\alpha,\beta) &=&\frac{\phi(q)^2}{q^2}\sum_{cd=q\atop (c,d)=1  }
\frac{\mu(c)
}{\phi(c)^2}\bigg(\mathcal{K}_{\alpha,\beta}(1)\frac{q}{2\pi}\zeta(1+\alpha+\beta)
\prod_{p\mid q}  \big(
1-\frac{1}{p^{1+\alpha+\beta}}\big)\\
&&\qquad + \mathcal{K}_{\alpha,\beta}(1-\alpha-\beta)
\big(\frac{q}{2\pi}\big)^{1-\alpha-\beta}\zeta(1-\alpha-\beta)
\prod_{p\mid q} \big( 1-\frac{1}{p^{1-\alpha-\beta}}\big)\bigg)\\
&& +2r_{\alpha,
\beta}\frac{\phi(q)}{q}\left(\frac{q}{2\pi}\right)^{1/2
-\beta}\sum_{cd=q\atop (c,d)=1, d>D  } \frac{\mu^2(c) }{\phi(c)^2}
\sumeven_{\psi \bmod c}\tau(\overline{\psi})\overline{\psi}(-d)
 L_q(\tfrac
12+\alpha,\psi)L_q(\tfrac 12 -\beta,\psi)
 \\
&&+2r_{\alpha,\beta}\frac{\phi(q)}{q}   \left(\frac{q}{2\pi}\right)^{1/2-\beta}
\sum_{cd=q, d\le D \atop (c,d)=1} \frac{\mu^2(c)}{\phi(c)\phi(d)}
\sumeven_{\psi \bmod d}\tau(\overline{\psi}) \overline{\psi}(c)
L_q(1/2+\alpha,\psi)L_q(1/2+\beta,\psi)
\\
&&\quad \qquad  \qquad \qquad +O\bigg(\frac{1}{q}\sum_{d\mid q\atop
d\le D} \phi(d)d^{3/2}+ \sqrt{q}\sum_{d\mid q \atop
d>D} \frac{\phi(d)}{d^{3/2}}+q^\epsilon +q^{\epsilon-1}D^2 \bigg)
\end{eqnarray*}
If we choose $D=q^{1/2}$ we see that the error-term is
$O(q^{1/4}\tau(q)).$
Actually, it is not hard to see that this O-term can also be written as
$$O((1+|\{d:d\mid q \mbox{ and } q^{1/2-\epsilon}<d<q^{1/2+\epsilon}\}|)~q^{1/4}).$$

Now to complete the argument we need to combine $M(\alpha,\beta)$
and $M(\beta,\alpha)$. We can accomplish this task by computing
$\mathcal{K}_{\alpha,\beta}(1)+\mathcal{K}_{ \beta,\alpha}(1)$ and
$\mathcal{K}_{\alpha,\beta}(1-\alpha-\beta )+\mathcal{K}_{
\beta,\alpha}(1-\alpha -\beta).$   We have
\begin{eqnarray*}\frac 12
\mathcal{K}_{\alpha,\beta}(1+\delta)&=&\int_0^\infty
K_{\alpha,\beta}(x) x^\delta \cos x~dx\\
&=& \int_0^\infty \frac{1}{2\pi i}\int_{(\tfrac
14)}F(s,\alpha,\beta) \Gamma(s+\tfrac 12 -\beta)\cos
\tfrac{\pi}{2}(s+\tfrac 12-\beta) x^{-s-\tfrac 12 +\beta} ~ds
~x^\delta \cos x ~dx\\
&=&\frac{1}{2\pi i}\int_{(\tfrac 14)}F(s,\alpha,\beta)
\Gamma(s+\tfrac 12 -\beta)\cos \tfrac{\pi}{2}(s+\tfrac 12-\beta)
\int_0^\infty   x^{-s-\tfrac 12 +\beta+\delta}\cos x~dx ~ds\\
&=&\frac{1}{2\pi i}\int_{(\tfrac 14)}F(s,\alpha,\beta)
\Gamma(s+\tfrac 12 -\beta)\cos \tfrac{\pi}{2}(s+\tfrac
12-\beta)\\
&&\qquad \qquad \qquad \times \Gamma(\tfrac 12-s+\beta+\delta)
\sin \tfrac{\pi}{2}(s+\tfrac 12 -\beta-\delta) ~ds.
\end{eqnarray*}

In  both situations $\delta=0$ and $\delta =-\alpha-\beta$ it is
easily checked that the integrand is odd with respect to the
change $(s,\alpha, \beta)\to (-s, \beta, \alpha)$. This means that
$$ \frac 12 \big(\mathcal{K}_{\alpha,\beta}(1+\delta)+\mathcal{K}_{
\beta,\alpha}(1+\delta)\big)$$ is equal to the residue at $s=0$ of
the integrand.  Since the residue of $F(s,\alpha, \beta)$ is 1, we
have
\begin{equation}\label{eq:Keval}  \mathcal{K}_{\alpha,\beta}(1+\delta)+\mathcal{K}_{
\beta,\alpha}(1+\delta) =2\Gamma( \tfrac 12 -\beta)\cos
\tfrac{\pi}{2}( \tfrac 12-\beta)  \Gamma(\tfrac 12 +\beta+\delta)
\sin \tfrac{\pi}{2}( \tfrac 12 -\beta-\delta).
\end{equation}
When $\delta=0$ this is easily seen to be $\pi$.  When
$\delta=-\alpha-\beta$ this expression evaluates to
$$2\Gamma(\tfrac12-\alpha)\Gamma(\tfrac12-\beta)\cos \tfrac{\pi}{2}(\tfrac
12-\alpha)
\cos\tfrac{\pi}{2}(\tfrac 12-\beta).$$

Note that for an even character $\chi$
\begin{eqnarray} \label{eq:X} \nonumber
X(1/2+\alpha,\chi)X(1/2+\beta,\overline{\chi})&=&4\frac{\tau(\chi)\tau(\overline{\chi})}{q^2}
\big(\frac{q}{2\pi}\big)^{1-\alpha-\beta}\Gamma(\tfrac 12
-\alpha)\Gamma(\tfrac12 -\beta) \cos \tfrac{\pi}{2}(\tfrac 12
-\alpha)\cos \tfrac{\pi}{2}(\tfrac 12 -\beta)\\
&=&\frac{4}{q}\nonumber
\big(\frac{q}{2\pi}\big)^{1-\alpha-\beta}\Gamma(\tfrac 12
-\alpha)\Gamma(\tfrac12 -\beta) \cos \tfrac{\pi}{2}(\tfrac 12
-\alpha)\cos \tfrac{\pi}{2}(\tfrac 12 -\beta)\\
&=:& X^+(q,\alpha,\beta).
\end{eqnarray}

 As a final step, we
prove that
$$\frac{\phi(q)^2}{q}\sum_{cd=q\atop (c,d)=1  } \frac{\mu(c)
}{\phi(c)^2}=  \sum_{d\mid q} \phi(d)\mu(q/d)=: \phi^*(q)$$ the
number  of primitive   characters modulo $q$. Since both sides are
multiplicative, it suffices to prove this identity for $q=p^r$ for
a prime $p$ and a positive integer $r$. If $r=1$ then both sides
are equal to
$$\frac{(p-1)^2}{p} \big( 1-\frac{1}{(p-1)^2}\big)=p-2.$$
If $r>1$ then the left side is
$$\frac{\phi(p^r)^2}{p^r}=\frac{(p^r-p^{r-1})^2}{p^r}=p^r-2p^{r-1}+p^{r-2}$$
and the right side is
$$ \phi(p^r)-\phi(p^{r-1})=p^r-2p^{r-1}+p^{r-2}$$
as desired.

Thus, we have proven

\begin{theorem} Let $m=\min\{c,d\}$ and $M=\max\{c,d\}$. Then, for any parameter $D$.
\begin{eqnarray*}
&&\sumstar_{\chi\bmod q\atop \chi(-1)=1}
L(1/2+\alpha,\chi)L(1/2+\beta,\overline{\chi})\\
&& \qquad =\frac{\phi^*(q)}{2}(\zeta_q(1+\alpha+\beta)   +
X^+(q,\alpha,\beta)\zeta_q(1-\alpha-\beta))
  \\
&&  \qquad \qquad +
 2\frac{\phi(q)}{q}\sum_{cd=q\atop (c,d)=1  } \frac{\mu^2(c) }{\phi(c)}\frac{1}{\phi(m)}
\sumeven_{\psi \bmod
m}\tau(\overline{\psi})\overline{\psi}(M) \\
&& \quad \qquad \qquad \times \bigg(\Gamma(1/2-\beta)\cos\tfrac{\pi}{2}
(\tfrac12-\beta)\big(\frac{q}{2\pi}\big)^{\tfrac 12 -\beta} L_q(\tfrac
12+\alpha,\psi)L_q(\tfrac 12 -\beta,\psi)  \\
&&\qquad \qquad \qquad \qquad + \Gamma(1/2-\alpha)\cos\tfrac{\pi}{2}
(\tfrac12-\alpha)\big(\frac{q}{2\pi}\big)^{\tfrac 12 -\alpha} L_q(\tfrac
12-\alpha,\psi)L_q(\tfrac 12 +\beta,\psi) \bigg)\\
&&\qquad \qquad \qquad \qquad \qquad +O\bigg(\frac{1}{q}\sum_{d\mid q\atop
d\le D} \phi(d)d^{3/2}+ \sqrt{q}\sum_{d\mid q \atop
d>D} \frac{\phi(d)}{d^{3/2}}+q^\epsilon +q^{\epsilon-1}D^2 \bigg)
\end{eqnarray*}
\end{theorem}

Letting $\alpha$ and $\beta\to 0$, we have
\begin{corollary} With $\phi^*(q)$ denoting the number of primitive characters modulo q, we have
\begin{eqnarray*}
\sumstar_{\chi\bmod q\atop \chi(-1)=1}
|L(1/2,\chi)|^2
&=&\frac{\phi^*(q)}{2}\frac{\phi(q)}{q}\left(\log\frac{q}{2\pi}+2\gamma+\frac{\Gamma'}{\Gamma}(1/2)
-\frac \pi 2 +2\sum_{p\mid q}\frac{\log p}{p-1}\right)
  \\
&&  \quad  +
 2\frac{\phi(q)}{q^{1/2}}\sum_{cd=q\atop (c,d)=1  } \frac{\mu^2(c) }{\phi(c)}\frac{1}{\phi(\min\{c,d\})}
\sumeven_{\psi \bmod
\min\{c,d\}}\tau(\overline{\psi})\overline{\psi}(\max\{c,d\}) L_q(\tfrac
12,\psi)^2 \\
&&\qquad \qquad +O\bigg(\frac{1}{q}\sum_{d\mid q\atop
d\le D} \phi(d)d^{3/2}+ \sqrt{q}\sum_{d\mid q \atop
d>D} \frac{\phi(d)}{d^{3/2}}+q^\epsilon +q^{\epsilon-1}D^2 \bigg)
\end{eqnarray*}
for any D.  With $D=\sqrt{q}$, the error term is $O(q^{1/4}d(q))$ where $d$ denotes the divisor function.
\end{corollary}

\begin{remark}
If we consider the mean square of odd primitive characters we get
the same expression for the main-terms only with $X^+(q,\alpha,
\beta)$ replaced by
$$X^-(q,\alpha,\beta)=X(1/2+\alpha,\chi)X(1/2+\beta,\overline{\chi})$$
where $\chi$ is an odd character; i.e.
$$X^-(q,\alpha,\beta)=\frac{4}{q}
\big(\frac{q}{2\pi}\big)^{1-\alpha-\beta}\Gamma(\tfrac 12
-\alpha)\Gamma(\tfrac12 -\beta) \sin \tfrac{\pi}{2}(\tfrac 12
-\alpha)\sin \tfrac{\pi}{2}(\tfrac 12 -\beta).$$
For comparison with [H-B], note that $\tfrac{\Gamma'}{\Gamma}(1/2)=-\gamma-\log 4$. Also, the result for odd characters
has the same main-term except that the $-\pi/2$ term appears with the opposite sign and so does not appear when
the sum is over all primitive characters. The other change is that the sum over $\psi \bmod \min\{c,d\}$ is now a sum
over the {\it odd} characters $\psi$ and it also has a factor of $-i$.

Explicitly,
\begin{eqnarray*}
\sumstar_{\chi\bmod q\atop \chi(-1)=-1}
|L(1/2,\chi)|^2
&=&\frac{\phi^*(q)}{2}\frac{\phi(q)}{q}\left(\log\frac{q}{2\pi}+2\gamma+\frac{\Gamma'}{\Gamma}(1/2)
+\frac \pi 2 +2\sum_{p\mid q}\frac{\log p}{p-1}\right)
  \\
&&  \quad  -
 2i\frac{\phi(q)}{q^{1/2}}\sum_{cd=q\atop (c,d)=1  } \frac{\mu^2(c) }{\phi(c)}\frac{1}{\phi(c)}
\sumodd_{\psi \bmod c}\tau(\overline{\psi})\overline{\psi}(-d) L_q(\tfrac
12,\psi)^2
\\
&&  \quad  -
 2i\frac{\phi(q)}{q^{1/2}}\sum_{cd=q\atop (c,d)=1  } \frac{\mu^2(c) }{\phi(c)}\frac{1}{\phi(d)}
\sumodd_{\psi \bmod d}\tau(\overline{\psi})\overline{\psi}(c) L_q(\tfrac
12,\psi)^2
 \\
&&\qquad \qquad +O\bigg(\frac{1}{q}\sum_{d\mid q\atop
d\le D} \phi(d)d^{3/2}+ \sqrt{q}\sum_{d\mid q \atop
d>D} \frac{\phi(d)}{d^{3/2}}+q^\epsilon +q^{\epsilon-1}D^2 \bigg)
\end{eqnarray*}
for any D.
\end{remark}

For all characters, the result is
\begin{corollary} \label{cor:all} With $\phi^*(q)$ denoting the number of primitive characters modulo q, we have
\begin{eqnarray*}
\sumstar_{\chi\bmod q}
&&|L(1/2,\chi)|^2
=\phi^*(q)\frac{\phi(q)}{q}\left(\log\frac{q}{2\pi}+2\gamma+\frac{\Gamma'}{\Gamma}(1/2)
 +2\sum_{p\mid q}\frac{\log p}{p-1}\right)
  \\
&&  \quad  +
 2\frac{\phi(q)}{q^{1/2}}\sum_{cd=q\atop (c,d)=1  } \frac{\mu^2(c) }{\phi(c)}\frac{1}{\phi(\min\{c,d\})}
\sum_{\psi \bmod
\min\{c,d\}}i(\psi)\tau(\overline{\psi})
(\delta_{c<d}\overline{\psi}(-d)+\delta_{c>d}\overline{\psi}(c)) L_q(\tfrac
12,\psi)^2 \\
&&\qquad \qquad +O\bigg(\frac{1}{q}\sum_{d\mid q\atop
d\le D} \phi(d)d^{3/2}+ \sqrt{q}\sum_{d\mid q \atop
d>D} \frac{\phi(d)}{d^{3/2}}+q^\epsilon +q^{\epsilon-1}D^2 \bigg)
\end{eqnarray*}
where $i(\psi)=-i$ if $\psi$ is odd and is 1 if $\psi $ is even, and where $D\le \sqrt{q}$;
also we have used the convention that $\delta_P$ is 1 if $P$ is true and 0 if $P$ is false.
\end{corollary}

\begin{remark}
Note that if $m$ is a prime, then the secondary main terms are of size
\begin{eqnarray*}
\approx \frac{q^{1/2}}{\phi(c)\sqrt{r_M}}\sqrt{m}\log m
\end{eqnarray*}
where $r_M$ is the least positive residue of $M$ modulo $m$.
For $c<\sqrt{q}$, this is
\begin{eqnarray*}
\approx \frac{q^{1/2}}{\phi(c)\sqrt{r_M}}\sqrt{c}\log c
\end{eqnarray*}
\end{remark}

\section{Remarks on Heath-Brown's proof}

We show how to simplify the end of Heath-Brown's argument.
Equation (12) of that paper gives that
\begin{eqnarray*}
\sum_{\chi \bmod q} |L(1/2,\chi)|^2=\frac{\phi(q)}{q}\sum_{k\mid q} \mu(q/k)T(k)
\end{eqnarray*}
where
\begin{eqnarray*}T(k)= 4 \sqrt{\frac{k}{2\pi}}\Re \sum_{n=1}^\infty \frac{d(n)}{\sqrt{n}} K\left(\frac{2\pi n}{k}\right)
\end{eqnarray*}
and
\begin{eqnarray*}
K(x)=e^{ix-i\pi/4}\cdot \frac{1}{2\pi i}\int_{(1)} \Gamma(s+1/2)  e^{-i\pi s/2}F(s)x^{-s}~ds
\end{eqnarray*}
with $F(s)=\frac{e^{s^2}\cos \pi s}{s}$.  Then $K$ satisfies the conditions (\ref{eq:Kcond}), so that we may write
\begin{eqnarray*}
T(k)=4\Re \frac{1}{2\pi i}\int_{(1)}\hat{K}(s) \zeta(s+1/2)^2\left(\frac{k}{2\pi }\right)^{s+1/2}~ds.
\end{eqnarray*}
By (\ref{eq:Kestimate}), we see that $\hat{K}(s)$ has poles at $s=0,-1,-2,\dots$ and is rapidly decaying in vertical strips
as $|\Im s|\to \infty$.  The double pole at $s=1/2$ from $\zeta(s+1/2)^2$ leads to the first main terms of
$T(k)$ and the lower order terms come from the poles of $\hat{K}$. We note that Heath-Brown records an asymptotic expansion
in powers of $k^{-1/2}$ whereas this argument shows that the expansion is actually in powers of $k^{-1}$.
Thus, the coefficients $c_n$ in Heath-Brown's theorem satisfy $c_n=0$ if $n$ is even.
The residue of $\hat{K}(s)$ at $s=0$ is $e^{-i\pi/4}\Gamma(1/2)$ and the values of $\hat{K}(1)$ and
$\hat{K}'(1)$ required for the first main terms can
be ascertained as in (\ref{eq:Keval}).

\section{Twisting}
In this section we briefly show how we can combine the theorem of this paper and that of Heath-Brown to
obtain a new result about twisting. Selberg [S] has shown that for positive integers $h$ and $k$,
coprime to each other and to $q$, and $s=\sigma+it$ and $s'=\sigma'+it'$ with $0<\sigma, \sigma' <1$,
\begin{eqnarray*} &&
\sumprime_{\chi \bmod q} L(s,\chi)L(s',\overline{\chi}) \chi(h)\overline{\chi}(k)
 =\frac{\phi(q)}{h^{s'}k^{s}}\zeta_q(s+s')\\
&&\qquad +\frac{\phi(q)^2}{q^{s+s'}h^{1-s}k^{1-s'}}\frac{(2\pi)^{s+s'-1}}{\pi}
\Gamma(1-s)\Gamma(1-s')\cos \tfrac \pi 2 (s-s')\zeta(2-s-s')\\
&& \qquad \qquad +O\left(\frac{|ss'|}{\sigma\sigma'(1-\sigma)(1-\sigma')}
\big(q^\epsilon(hq^{1-\sigma}+kq^{1-\sigma'}+hkq^{1-\sigma-\sigma'})\big)\right);
\end{eqnarray*}
and
\begin{eqnarray*} &&
\sumprime_{\chi \bmod q} L(s,\chi)L(s',\overline{\chi}) \chi(h)\overline{\chi}(-k)
 \\&&\qquad=-\frac{\phi(q)^2}{q^{s+s'}h^{1-s}k^{1-s'}}\frac{(2\pi)^{s+s'-1}}{\pi}
\Gamma(1-s)\Gamma(1-s')\cos \tfrac \pi 2 (s+s')\zeta(2-s-s')\\
&& \qquad \qquad +O\left(\frac{|ss'|}{\sigma\sigma'(1-\sigma)(1-\sigma')}
\big(q^\epsilon(hq^{1-\sigma}+kq^{1-\sigma'}+hkq^{1-\sigma-\sigma'})\big)\right);
\end{eqnarray*}

here the sum is over all non-principal characters.

Iwaniec and Sarnak [IS] have shown, in very simple fashion, that
\begin{eqnarray*}
&&\sumstar_{\chi \bmod q\atop \chi(-1)=1} |L(1/2,\chi)|^2 \chi(h) \overline{\chi}(k)
=\frac{\phi^*(q)\phi(q)}{2q\sqrt{hk}}\left(\log\frac {q}{2\pi hk}+2\gamma-\frac{\pi}2 +\frac{\Gamma'}{\Gamma}(1/2)
+2\sum_{p\mid q} \frac{\log p}{p-1}\right)\\
&&\qquad   +O(\beta(h,k))
\end{eqnarray*}
where
\begin{eqnarray*}
\beta(h,k)=\sum_{hn_1\ne k n_2}(hn_1\pm kn_2,q)|W(\pi n_1n_2/q)| (n_1n_2)^{-1/2}
\end{eqnarray*}
with $W(y)\ll (1+y)^{-1}$. While they don't give
a specific bound for $\beta(h,k)$, they observe that
\begin{eqnarray*}
\sum_{h,k\le M} \frac{\beta(h,k)}{\sqrt{hk}}\ll \tau(q)\sqrt{q}M \log^4 Mq.
\end{eqnarray*}
This average bound cannot be obtained from Selberg's pointwise bound.

It turns out that we can combine the theorem of this paper with that of Heath-Brown's paper
to obtain some results along the lines of those of Selberg and of Iwaniec-Sarnak.
Suppose, for simplicity, that $p$ and $h$ are both primes, with $h<p,$ and
that we are interested in
\begin{eqnarray*}
S(p,h):=\sumstar_{\chi \bmod p} |L(1/2,\chi)|\chi(h).
\end{eqnarray*}
Note that
$$\overline{S(p,h)}=\sumstar_{\chi \bmod p} |L(1/2,\chi)|\overline{\chi(h)}=S(p,h)$$
by replacing $\overline{\chi}$ by $\chi$. Thus, $S(p,h)$ is real.

By Theorem 1, we have
\begin{eqnarray*}
S_1=\sum_{\chi \bmod ph}|L(1/2,\chi)|^2=\frac{\phi(ph)}{ph}\big(T(ph)-T(p)-T(h)+T(1)\big)
\end{eqnarray*}
where, for certain constants $c_\ell$,
\begin{eqnarray*}
T(k)=k\log k +Ak +Bk^{1/2}+\sum_{\ell=0}^{L} c_\ell k^{-\ell/2}+O(k^{-(L+1)/2}).
\end{eqnarray*}
Thus,
\begin{eqnarray}  \label{eq:S1_1}
\nonumber S_1&=&ph \log ph -2p\log p-2h\log h-p\log h-h\log p+\frac p h \log p+A(hp-2p-2h+\frac ph)\\
&& \qquad +B\sqrt{p}(\sqrt{h}-1-1/\sqrt{h}+1/h)+2\log p -\frac  {\log p}  {h}+ O(\sqrt{h})\\
&=&\nonumber ph \log ph -2p\log p-2h\log h-p\log h-h\log p+\frac p h \log p+A(hp-2p-2h+\frac ph)\\
&& \qquad \nonumber +B\sqrt{p}(\sqrt{h}-1-1/\sqrt{h})+ O(\sqrt{h}+\log p + \sqrt{p}/h).
\end{eqnarray}

Now, in terms of primitive characters, we have
\begin{eqnarray*}
&&S_1=\sumstar_{\chi \bmod ph}|L(1/2,\chi)|^2+\sumstar_{\chi\bmod p}|L(1/2,\chi)|^2\left|1-\frac{\chi(h)}
{\sqrt{h}}\right|^2 +\sumstar_{\chi\bmod h}|L(1/2,\chi)|^2\left|1-\frac{\chi(p)}
{\sqrt{p}}\right|^2 \\ \nonumber
&&\qquad \qquad\qquad +\zeta(1/2)^2(1-1/\sqrt{h})^2(1-1/\sqrt{p})^2.
\end{eqnarray*}

By Corollary \ref{eq:corapp} with $h<D<p$, we have,
\begin{eqnarray} \label{eq:corapp}
\sumstar_{\chi\bmod ph}\nonumber
&&|L(1/2,\chi)|^2
=\phi^*(ph)\frac{\phi(ph)}{ph}\left(\log ph +A
 +2\frac{\log p}{p-1}+2\frac{\log h}{h-1}\right)
  \\
&& \qquad      +
 2\frac{\phi(ph)}{(ph)^{1/2}}\left(\zeta_q(1/2)^2+ \frac{1}{h-1}
\sum_{\psi \bmod
h}i(\psi)\tau(\overline{\psi})\overline{\psi}(p) \bigg(\frac{1}{p-1}+\frac{\psi(-1)}{h-1}\bigg) L_p(\tfrac
12,\psi)^2 \right)\\
&&\qquad \qquad \qquad +O(\sqrt{h})\nonumber .
\end{eqnarray}
The first term on the right of (\ref{eq:corapp}) is
\begin{eqnarray*}
&&(p-3)\bigg(h-3+\frac 2 h\bigg)\bigg(\log ph+A +2\frac{\log p}{p}+2\frac{\log h}{h-1}\bigg)+O(\log h) \\
&&\qquad = p\bigg(h-3+\frac 2 h\bigg)\bigg(\log ph+A +2\frac{\log p}{p}+2\frac{\log h}{h-1}\bigg) \\
&&\qquad \qquad - 3 \bigg(h-3+\frac 2 h\bigg)\log p-3h\log h-3Ah +O(\log h)\\
&=& hp\log hp -h\log p-p\log h-3p\log p-3h\log h+A(hp-3h-3p+2p/h)\\
&&\qquad +2\frac p h \log \frac ph +O(\log p)
\end{eqnarray*}
For $\psi$ primitive modulo $h$, the functional equation implies that
\begin{eqnarray*}
\tau(\overline{\psi}) L(1/2,\psi)=h^{1/2}L(1/2,\overline{\psi})\left\{\begin{array}{ll}
1&\mbox{ if } \psi \mbox{ even}\\i &\mbox{ if } \psi \mbox{ odd}\end{array}      \right.
.\end{eqnarray*}

Hence, the second term on the right side of (\ref{eq:corapp}) is equal to
\begin{eqnarray*}
 B\sqrt{p}~(\sqrt{h}-2)+2\frac{\sqrt{p}}{h-1}S(h,-p)
+O(\sqrt{h}+\sqrt{p}/h).
\end{eqnarray*}

Using Theorem 1 we find that
\begin{eqnarray*}
\sumstar_{\chi\bmod p}
|L(1/2,\chi)|^2\left|1-\frac{\chi(h)}
{\sqrt{h}}\right|^2
&=&\bigg(1+\frac 1h\bigg)(p\log p +p A-\log p+B\sqrt{p})-2 \frac{S(p,h)}{\sqrt{h}}+O(1)
\end{eqnarray*}
and
\begin{eqnarray*}
\sumstar_{\chi\bmod h}|L(1/2,\chi)|^2\left|1-\frac{\chi(p)}
{\sqrt{p}}\right|^2
&=&h\log h+h A+O(\sqrt{h}).
\end{eqnarray*}

Thus,
\begin{eqnarray}  \label{eq:S1_2}
S_1&=&hp\log hp -h\log p-p\log h-3p\log p-3h\log h+A(hp-3h-3p+2p/h)\\
&&\qquad +2\frac p h \log \frac ph + B\sqrt{p}~(\sqrt{h}-2\nonumber)+2\frac{\sqrt{p}}{h-1}
S(h,-p)-2 \frac{S(p,h)}{\sqrt{h}}\\
&&\qquad + \nonumber
\bigg(1+\frac 1h\bigg)(p\log p +p A+B\sqrt{p}) +h\log h+hA+O(\sqrt{h}+\log p+\sqrt{p}/h ).\nonumber
\end{eqnarray}

 Combining   (\ref{eq:S1_1}) and (\ref{eq:S1_2}), we have
\begin{theorem}
\begin{eqnarray*}
&&S(p,h)=
\frac{\sqrt{p}}{\sqrt{h}}S(h,-p)+
\frac p {\sqrt {h}} \left(\log \frac p h +A \right) +\frac B 2\sqrt{p}+O(h+\log p+\sqrt{p/h}\log p)
\end{eqnarray*}
uniformly for $h<p$.
\end{theorem}

When $h=O(\sqrt{p})$ the second term on the right is the dominating one. However, for larger $h$,
the first term is potentially
larger;
note that $S(h,-p)$ can be as large as $h\log h$; in fact if $-p\equiv 1 \bmod h$ it will be this large;
in this case, one has an asymptotic formula for $S(p,h)$ for $h$ all the way up to $p$. In any event,
the theorem determines the asymptotic behavior of $S(p,h)-{\sqrt{p/h}}S(h,-p)$ for $h$ as large as $p^{2/3}$,
and illustrates a kind of `reciprocity formula' between $S(h,p)$ and $S(p,-h)$. It would be worth exploring
such relationships further.

\end{document}